\documentclass[12pt,a4paper]{amsart}
\usepackage{amsmath,amssymb,color,multicol,setspace}
\usepackage{hyperref}
\usepackage[utf8,utf8x]{inputenc}
\usepackage{fullpage}

\usepackage{longtable}
\usepackage{graphicx}
\usepackage{pict2e}
\usepackage{xy,amscd}
\usepackage{epsfig}
\usepackage{rotating}
\usepackage{xspace}
\xyoption{all}
\usepackage{color}

\usepackage{enumitem}

\newtheorem{propo}{Proposition}[section]
\newtheorem{corol}[propo]{Corollary}

\newtheorem{lemma}[propo]{Lemma}

\newtheorem{algor}[propo]{Algorithm}

\theoremstyle{definition}
\newtheorem{defin}[propo]{Definition}
\newtheorem{examp}[propo]{Example}

\newtheorem{oppro}[propo]{Open Problem}

\theoremstyle{remark}
\newtheorem{remar}[propo]{Remark}

\newcommand{\NN }{\mathbb{N}}
\newcommand{\CC }{\mathbb{C}}
\newcommand{\RR }{\mathbb{R}}
\newcommand{\FF }{\mathbb{F}}
\newcommand{\QQ }{\mathbb{Q}}
\newcommand{\ZZ }{\mathbb{Z}}
\newcommand{\PP }{\mathbb{P}}

\DeclareMathOperator{\Aut}{Aut}

\DeclareMathOperator{\rk}{rk}

\newcommand{\Ac }{\mathcal{A}}
\newcommand{\Bc }{\mathcal{B}}

\newcommand{\Kc }{\mathcal{K}}
\newcommand{\Vc }{\mathcal{V}}

\newcommand{\Uc}{{\mathcal U}}

\DeclareMathOperator{\PGL}{PGL}

\DeclareMathOperator{\PG}{PG}
\newcommand{\PFq }{\PG(2,q)}

\newcommand{\algo}[6]
{
\begin{algor}{{\tt #1}{\rm (#2)}}\label{#1}\end{algor}
\vspace{-6pt}\noindent{\it #3}.

{\bf Input:} #4

{\bf Output:} #5

\newcounter{#1}
\begin{list}{\textbf{\arabic{#1}.}}{\usecounter{#1}}
#6\end{list}\vspace{3pt}}

\definecolor{darkgreen}{rgb}{0.0,0.1,0.6}
\newcommand{\df}[1]{{\bf\color{darkgreen} #1}}

\title[A greedy algorithm to compute arrangements of lines in the projective plane]
{A greedy algorithm to compute arrangements of lines in the projective plane}

\author{Michael~Cuntz}
\address{Michael Cuntz, Leibniz Universit\"at Hannover,
Institut f\"ur Algebra, Zah\-lentheorie und Diskrete Mathematik,
Fakult\"at f\"ur Mathematik und Physik,
Wel\-fengarten 1,
D-30167 Hannover, Germany}
\email{cuntz@math.uni-hannover.de}

\begin{document}


\begin{abstract}
We introduce a greedy algorithm optimizing arrangements of lines with respect to a property. We apply this algorithm to the case of simpliciality: it recovers all known simplicial arrangements of lines in a very short time and also produces a yet unknown simplicial arrangement with 35 lines. We compute a (certainly incomplete) database of combinatorially simplicial complex arrangements of hyperplanes with up to 50 lines. Surprisingly, it contains several examples whose matroids have an infinite space of realizations up to projectivities.
\end{abstract}

\maketitle

\section{Introduction}

A simplicial arrangement is a finite set of linear hyperplanes in a real vector space which decomposes its complement into open simplicial cones, cf.\ \cite{a-Melchi41}.
A classification of simplicial arrangements, even in the case of dimension three, has not been achieved in full generality yet. There is a topological result by Deligne \cite{MR0422673} and there are some classifications of smaller classes, as in \cite{p-CH09c}, \cite{p-CH10}, or \cite{CM17}. But until now, no explicit approach to a classification is known. In this early stage of investigations, it is common to collect examples as in \cite{AW86a}, \cite{p-G-09} and \cite{p-C12}.

In many areas of mathematics, examples are mainly used as counter-examples in order to demonstrate that certain propositions do not hold. When dealing with discrete structures however, one often encounters a finite set of exceptions. For instance, the discovery of some of the finite simple groups has been celebrated although each such group is ``merely'' an example. Reflection groups are a further example of a structure with sporadic cases in which the situation is less difficult (a classification of finite real reflection groups is even accessible to students).
The situation is, albeit less prominent, apparently similar in the case of simplicial arrangements (note that real reflection groups ``are'' very special simplicial arrangements). For the case of rank three it is conjectured that there is, apart from three infinite series, only a finite number of sporadic examples.
This is why a classification will ultimately probably be found via a combination of theoretical arguments and a collection of examples.

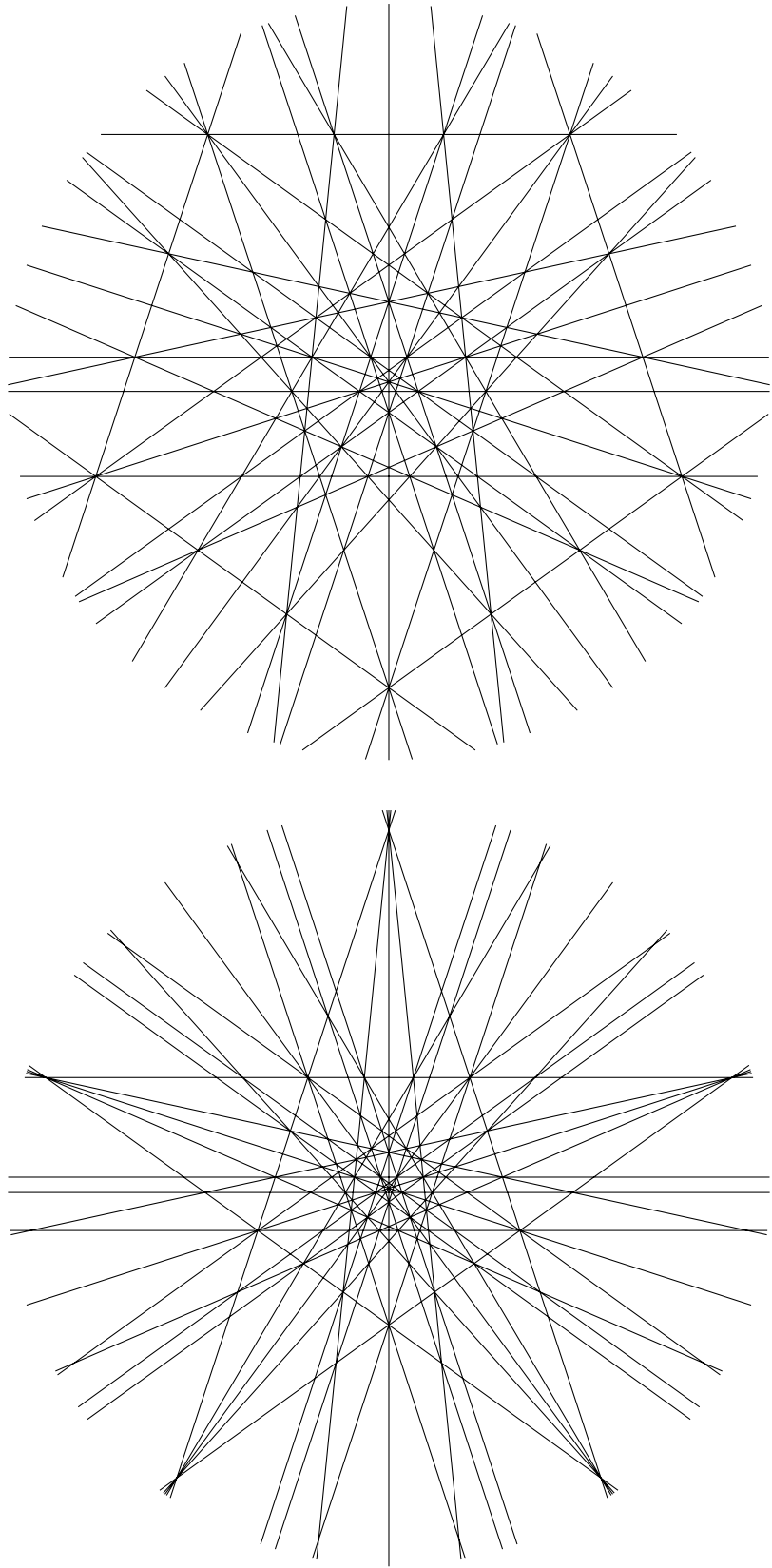
\begin{figure}
\begin{center}
\setlength{\unitlength}{0.75pt}
\begin{picture}(600,400)(0,200)
\moveto(300.000000000000000000000000000,600.000000000000000000000000000)
\lineto(300.000000000000000000000000000,200.000000000000000000000000000)
\moveto(490.211222793226967798729969454,338.196353477279565667920546203)
\lineto(109.788777206773032201270030546,461.803646522720434332079453798)
\moveto(493.649167310370844258963269989,350.000000000000000000000000000)
\lineto(106.350832689629155741036730011,350.000000000000000000000000000)
\moveto(499.274358232042334362419174016,382.978538511406797288194095982)
\lineto(254.610029286908278125665503943,205.218710963643441224147500799)
\moveto(398.968878932913366486834222682,226.203679547689109239257945313)
\lineto(139.100807178870402956358138055,518.791623229539876266538177587)
\moveto(345.389970713091721874334496057,205.218710963643441224147500799)
\lineto(100.725641767957665637580825984,382.978538511406797288194095982)
\moveto(277.913022855087308595830707670,598.776672274691489802789654241)
\lineto(239.653522778648111312643178333,209.321467681931069380481648846)
\moveto(471.221167970164474305630842449,296.639893387570835412029491694)
\lineto(377.767768142817375428610303583,584.261157703095415405043562190)
\moveto(360.346477221351888687356821667,209.321467681931069380481648846)
\lineto(322.086977144912691404169292330,598.776672274691489802789654241)
\moveto(460.899192821129597043641861945,518.791623229539876266538177587)
\lineto(201.031121067086633513165777317,226.203679547689109239257945313)
\moveto(222.232231857182624571389696417,584.261157703095415405043562190)
\lineto(128.778832029835525694369157551,296.639893387570835412029491694)
\moveto(490.211222793226967798729969454,461.803646522720434332079453798)
\lineto(109.788777206773032201270030546,338.196353477279565667920546203)
\moveto(482.222484170879662643624102652,482.431585345627744919227389245)
\lineto(100.005851715243812024145561362,398.470081095301612074296060614)
\moveto(499.994148284756187975854438639,398.470081095301612074296060614)
\lineto(117.777515829120337356375897348,482.431585345627744919227389245)
\moveto(363.256211211645214920412850329,589.733106608065979379397554340)
\lineto(165.293579024573374624903725434,252.168406123754649904854125661)
\moveto(453.753655398313702491477621582,272.094513598295559071308500372)
\lineto(130.842610790189231141880033800,506.704159598961254682587802403)
\moveto(417.556709601146281311756946962,238.196353478199504079312614697)
\lineto(182.443290398853718688243053038,561.803646521800495920687385304)
\moveto(417.556709601146281311756946962,561.803646521800495920687385304)
\lineto(182.443290398853718688243053038,238.196353478199504079312614697)
\moveto(451.210929116832651774845355991,530.901699437494742410229341718)
\lineto(148.789070883167348225154644009,530.901699437494742410229341718)
\moveto(434.706420975426625375096274565,252.168406123754649904854125661)
\lineto(236.743788788354785079587149671,589.733106608065979379397554340)
\moveto(469.157389209810768858119966200,506.704159598961254682587802403)
\lineto(146.246344601686297508522378418,272.094513598295559071308500372)
\moveto(499.570305045734574920500586581,413.103180680752626561830962120)
\lineto(100.429694954265425079499413419,413.103180680752626561830962120)
\moveto(366.543788404626644659212896707,588.605207310827652327105312872)
\lineto(242.976195492496265797777118211,208.301576116312326093823038893)
\moveto(495.873049068606738248705337535,440.419656710159934136108356336)
\lineto(137.302356105043031298512243909,283.683721384194157968493747394)
\moveto(357.023804507503734202222881789,208.301576116312326093823038893)
\lineto(233.456211595373355340787103293,588.605207310827652327105312872)
\moveto(462.697643894956968701487756091,283.683721384194157968493747394)
\lineto(104.126950931393261751294662465,440.419656710159934136108356336)
\moveto(427.275940655054453355676367877,554.275192206560082985269812989)
\lineto(113.945453650666516186871816537,326.626259583256092456301078668)
\moveto(374.132190489614720405654246783,214.246350417518103611588834661)
\lineto(250.791450925404741463669320393,593.851795704793898421202948099)
\moveto(349.208549074595258536330679607,593.851795704793898421202948099)
\lineto(225.867809510385279594345753217,214.246350417518103611588834661)
\moveto(499.937365889193711502325840010,394.995030340683481761242421658)
\lineto(100.062634110806288497674159990,394.995030340683481761242421658)
\moveto(486.054546349333483813128183463,326.626259583256092456301078668)
\lineto(172.724059344945546644323632123,554.275192206560082985269812989)
\moveto(312.288004461537720802280578186,200.377844550377552366159024401)
\lineto(192.606684765979341077414164682,568.720703658579690813478532008)
\moveto(458.810628018575591451615476229,521.569669031159303028451424212)
\lineto(135.305669219088494986975271940,286.528517200894456737667115587)
\moveto(464.694330780911505013024728060,286.528517200894456737667115587)
\lineto(141.189371981424408548384523771,521.569669031159303028451424212)
\moveto(407.393315234020658922585835318,568.720703658579690813478532008)
\lineto(287.711995538462279197719421814,200.377844550377552366159024401)
\strokepath
\end{picture}
\medskip

\begin{picture}(600,400)(0,200)
\moveto(300.000000000000000000000000000,600.000000000000000000000000000)
\lineto(300.000000000000000000000000000,200.000000000000000000000000000)
\moveto(490.211222793226967798729969454,338.196353477279565667920546203)
\lineto(109.788777206773032201270030546,461.803646522720434332079453797)
\moveto(498.746069143517905200435966577,377.639320225002103035908263313)
\lineto(101.253930856482094799564033423,377.639320225002103035908263313)
\moveto(489.126237694575395281647151622,465.048183798588634174533758617)
\lineto(179.692952539083360028222928919,240.230746602367979109499556115)
\moveto(418.394000371230172483882828924,238.808000582854144621459012110)
\lineto(153.909921284379666800828598457,536.593151002763879050473110876)
\moveto(420.307047460916639971777071081,240.230746602367979109499556115)
\lineto(110.873762305424604718352848378,465.048183798588634174533758617)
\moveto(301.036867533219357422017118763,599.997312246236338276637921351)
\lineto(262.097770950161753021245257506,203.624286040626709838542003172)
\moveto(414.771984881876253742183930756,236.209305861796993685616675879)
\lineto(296.579252476919680850238654011,599.970744076185654076482179016)
\moveto(337.902229049838246978754742494,203.624286040626709838542003172)
\lineto(298.963132466780642577982881237,599.997312246236338276637921351)
\moveto(446.090078715620333199171401543,536.593151002763879050473110876)
\lineto(181.605999628769827516117171076,238.808000582854144621459012110)
\moveto(303.420747523080319149761345990,599.970744076185654076482179016)
\lineto(185.228015118123746257816069244,236.209305861796993685616675879)
\moveto(490.211222793226967798729969454,461.803646522720434332079453797)
\lineto(109.788777206773032201270030546,338.196353477279565667920546203)
\moveto(490.529102832561270924157216790,460.816617579566957437200944278)
\lineto(101.523167361054434227924844156,375.363707551419345883184475639)
\moveto(498.476832638945565772075155844,375.363707551419345883184475639)
\lineto(109.470897167438729075842783210,460.816617579566957437200944278)
\moveto(384.762754489706792951253177327,581.149870138837502049544365776)
\lineto(183.283740302606322050382282271,237.589056026847266426020708141)
\moveto(458.289305626776213588845524687,277.752318123438299169980430930)
\lineto(134.821935191259912015732769274,512.766160287737292988804766655)
\moveto(417.556709601146281311756946962,238.196353478199504079312614697)
\lineto(182.443290398853718688243053038,561.803646521800495920687385304)
\moveto(417.556709601146281311756946962,561.803646521800495920687385304)
\lineto(182.443290398853718688243053038,238.196353478199504079312614697)
\moveto(491.240552751959901167429768726,458.541019662496845446137605031)
\lineto(108.759447248040098832570231274,458.541019662496845446137605031)
\moveto(416.716259697393677949617717729,237.589056026847266426020708141)
\lineto(215.237245510293207048746822673,581.149870138837502049544365776)
\moveto(465.178064808740087984267230726,512.766160287737292988804766655)
\lineto(141.710694373223786411154475313,277.752318123438299169980430930)
\moveto(499.914134895983556180183984515,405.859920544724968641595545935)
\lineto(100.085865104016443819816015485,405.859920544724968641595545935)
\moveto(363.928013872447184450395303815,589.507807338706233947917142417)
\lineto(240.329452355286854023140114105,209.108864156084978186409133419)
\moveto(489.888229691498476497755199765,462.789013566297698328392732910)
\lineto(124.948262538377343417835354455,303.268985264977525718268205141)
\moveto(359.670547644713145976859885895,209.108864156084978186409133419)
\lineto(236.071986127552815549604696185,589.507807338706233947917142417)
\moveto(475.051737461622656582164645545,303.268985264977525718268205141)
\lineto(110.111770308501523502244800235,462.789013566297698328392732910)
\moveto(447.645408726738306622056275003,534.910463945219563872989631338)
\lineto(126.067999682248866657180943165,301.269765190870635457131194954)
\moveto(367.349713415765591576045431485,211.681078744555664626913911066)
\lineto(243.796477874396921362752702853,591.940522299687527961132499296)
\moveto(356.203522125603078637247297147,591.940522299687527961132499296)
\lineto(232.650286584234408423954568515,211.681078744555664626913911066)
\moveto(499.987474747149763141761595232,397.761709523288860318522162157)
\lineto(100.012525252850236858238404768,397.761709523288860318522162157)
\moveto(473.932000317751133342819056834,301.269765190870635457131194954)
\lineto(152.354591273261693377943724997,534.910463945219563872989631338)
\moveto(340.149475960525256648410753323,204.071391624155089238464013038)
\lineto(217.318136057771766330075942543,582.109059014204064526924231562)
\moveto(460.477382197636757919721127788,519.360838649004209169024194035)
\lineto(136.891345935165963670467485463,284.260780332860809469746160975)
\moveto(463.108654064834036329532514537,284.260780332860809469746160975)
\lineto(139.522617802363242080278872212,519.360838649004209169024194035)
\moveto(382.681863942228233669924057457,582.109059014204064526924231562)
\lineto(259.850524039474743351589246677,204.071391624155089238464013038)
\strokepath
\end{picture}
\end{center}
\caption{The ``new'' simplicial arrangement of rank three with $35$ hyperplanes (from two different perspectives).\label{sim35}}
\end{figure}

In this paper, we introduce an algorithm to approximate arrangements of lines with respect to a given property. This algorithm works surprisingly well in the case of simpliciality, since its implementation finds all known simplicial arrangements (with up to $50$ lines) in a few minutes on an ordinary computer. Using a computer cluster we even find a yet unknown simplicial arrangement of lines with $35$ lines (Fig.\ \ref{sim35}).
Moreover, this algorithm may also be used to attack other open problems as for example Terao's conjecture, where one difficulty is to find matroids with an infinite moduli space of realizations. Again, our implementation finds all known prominent examples (free but not recursively free or with similar properties) within a short time.

The algorithm is based on the following intuition. Simplicial arrangements of lines, like most ``interesting'' arrangements (as for example those considered in the context of Terao's conjecture), have few double points. In fact, it even turns out that simpliciality is closely related to the property to have few double points, as demonstrated for example in \cite{p-GT-13}, where it is shown that asymptotically, an arrangement of lines with few double points is near to be a simplicial arrangement belonging to one of the infinite series. Moreover, it is easy to associate an invariant in $\ZZ$ to each matroid of rank three which is zero if and only if any realization of the matroid is simplicial and which quantifies how ``far'' it is from being simplicial.

Now the key idea in the algorithm is: take an arrangement of lines, remove a line, and replace this line by a line through two intersection points of the arrangement. This will often reduce the number of double points. If the new arrangement ``improves'' the chosen invariant, then discard the old arrangement and repeat the procedure with the new one until the invariant is zero.
This idea alone is not sufficient to obtain all the desired examples. In this primitive version, an arrangement will tend to become rational during the procedure (see Remark \ref{remQ} for an explanation). It is thus important to include algebraic numbers or possibly transcendents if requested (see Remark \ref{remtrans}).
Further technical improvements which are necessary to produce a working implementation are discussed in Section \ref{sectec}.

As a result we present several yet unknown arrangements of lines:
\begin{itemize}
\item We find a ``new'' real simplicial arrangement of rank three with $35$ lines.
\item We collect a database with $1318$ combinatorially simplicial arrangements of lines with up to $50$ lines over $\CC$.
\item This database includes several matroids of rank three which are combinatorially simplicial and have infinite moduli space in characteristic zero.
\end{itemize}

In Sections \ref{sec:simp} and \ref{matmod} we recall all required notions on simplicial arrangements and moduli spaces of matroids of rank three, including some (maybe new) open problems.
Section \ref{sectec} is devoted to the description of the algorithms, results of our implementations are collected in Section \ref{sec:resu}.

\medskip
\noindent{\bf Acknowledgement:}
{The computations required for the results of this paper were performed on a computer cluster funded by the DFG, project number 411116428.}

\section{Simplicial arrangements}\label{sec:simp}

The main application of our algorithm is to produce simplicial arrangements, so let us recall the basic notions in this section.

\begin{defin}\label{A_R}
Let $K$ be a field, $r\in\NN$, and $V:=K^r$.
An \df{arrangement of hyperplanes} $(\Ac,V)$ (or $\Ac$ for short) is a finite set of hyperplanes $\Ac$ in $V$.
It is \df{central} if all elements of $\Ac$ are linear subspaces and \df{essential} if $\bigcap_{H\in\Ac} H=0$.
\end{defin}

\begin{defin}\label{A_R_2}
Let $r\in\NN$, $V:=\RR^r$, and $\Ac$ an arrangement in $V$.
Let $\Kc(\Ac)$ be the set of connected components (\df{chambers}) of $V\backslash \bigcup_{H\in\Ac} H$.
If every chamber $K$ is an \df{open simplicial cone}, i.e.\ there exist
$\alpha^\vee_1,\ldots,\alpha^\vee_r \in V$ such that
\begin{equation*}
K = \Big\{ \sum_{i=1}^r a_i\alpha^\vee_i \mid a_i> 0 \quad\mbox{for all}\quad
i=1,\ldots,r \Big\} =: \langle\alpha^\vee_1,\ldots,\alpha^\vee_r\rangle_{>0},
\end{equation*}
then $\Ac$ is called a \df{simplicial arrangement}.
\end{defin}

\begin{examp}
\begin{enumerate}
\item Figure \ref{sim35} displays an example for $r=3$, displayed in the real projective plane. Simpliciality of the chambers translates to the fact that all regions are triangles in the picture.
\item Let $W$ be a real reflection group, $R\subseteq V^*$ the set of roots of $W$.
Then $\Ac = \{\ker \alpha \mid \alpha\in R\}$ is a simplicial arrangement.
\end{enumerate}
\end{examp}

Since a hyperplane is uniquely determined by a linear form up to scalars, i.e.\ a one dimensional subspace of the dual space, it is often easier to work in the projective space instead of $V$.
We will mostly concentrate on the case of rank three, thus we are working with lines in the projective plane.
We denote $\PP_2 K$ the projective plane over $K$ and $\PFq=\PP_2\FF_q$. Moreover, we will sometimes denote both projective lines and points with coordinates $(a:b:c)$ since points and lines are dual to each other in the plane. So it makes sence to write $(a:b:c)$, $a,b,c\in K$ for a hyperplane in an arrangement of rank three over $K$.

\begin{defin}
Let $\Ac$ be an arrangement.
For $X\le V$, we define the \df{localization}
$$
\Ac_X := \{ H \in \Ac \mid X \subseteq H \}
$$ 
of $\Ac$ at $X$, and the \df{restriction} $(\Ac^X,X)$ of $\Ac$ to $X$, where 
$$
\Ac^X := \{ X\cap H \mid H \in \Ac \setminus \Ac_X \}.
$$
\end{defin}

\begin{defin}
The \emph{intersection lattice} $L(\Ac)$ of $\Ac$ consists of all intersections of elements of $\Ac$
including $V$ as the empty intersection. 
The \emph{rank} $\rk(\Ac)$ of $\Ac$ is defined as the codimension of the intersection of all hyperplanes in $\Ac$.
For $0 \leq k \leq r$ we write $L_k(\Ac) := \{ X \in L(\Ac) \mid r(X) = k \}$.
\end{defin}

\begin{remar}
If $\Ac$ is simplicial, then all localizations and restrictions to elements of its intersection lattice are simplicial.
\end{remar}

\begin{propo}[e.g.\ {\cite[2]{p-CG-13}}]
Let $\Ac$ be a central essential arrangement of hyperplanes in $\RR^r$, $r\ge 2$. Then $\Ac$ is simplicial if and only if
\begin{equation}\label{combsim0}
r |\Kc(\Ac)| = 2 \sum_{H\in\Ac} |\Kc(\Ac^H)|.
\end{equation}
\end{propo}

\begin{remar}
By Zaslavsky's theorem, $|\Kc(\Ac)| = (-1)^r\chi_\Ac(-1)$ which depends only on the intersection lattice of $\Ac$.
Thus simpliciality is a purely combinatorial property.
\end{remar}
In this article the following equivalent formulation is more convenient \footnote{Notice the constant $3$ which comes from the Euler characteristic of the sphere.}.

\begin{corol}\label{corcombsim}
Let $\Ac$ be a central essential arrangement in $V=K^3$.
Then $\Ac$ is simplicial if and only if $\sigma(\Ac)=0$, where
\begin{equation}\label{combsim}
\sigma(\Ac) = \sigma(L(\Ac)) := 3 + \sum_{v\in L_2(\Ac)} (|\Ac_v|-3) \in \ZZ.
\end{equation}
\end{corol}

It is interesting to extend the definition of simpliciality to arbitrary arrangements although the original motivation using chambers is lost.

\begin{defin}
Let $K$ be a field and $\Ac$ an arrangement of hyperplanes in $K^3$. Then $\Ac$ is (\df{combinatorially}) \df{simplicial} if and only if $\sigma(L(\Ac))=0$.
\end{defin}

For example, simplicial arrangements over $\CC$ have many other nice properties and seem to be rare like real simplicial arrangements.

\section{Matroids and moduli spaces}
\label{matmod}

\begin{defin}[c.f.\ \cite{p-ACKN-14}]\label{modulispace} 
Let $K$ be a field and $\Ac=\{H_1,\dotsc,H_\ell\}$ be 
a central arrangement in $K^n$ ordered by the indices of the hyperplanes.
To a matrix $M=[m_1,\dotsc,m_\ell]\in K^{n\times\ell}$, 
we attach a central arrangement 
$\Bc_M=\{H_i'=\ker(m_i)\mid 1\leq i\leq\ell\}\setminus\{K^n\}$ 
in $K^n$. 
Consider the following condition for $M$: 
\begin{eqnarray*}
(\ast) && |\Bc_M|=\ell \text{ and there exists an isomorphism } \pi\colon L(\Ac)\rightarrow L(\Bc_M)\\
&& \text{ of graded lattices such that }
\pi(H_i)=H_i' \text{ for } 1\leq i\leq\ell.
\end{eqnarray*} 
For a lattice $L$ on $\{1,\ldots,\ell\}$, Yuzvinsky \cite{p-sY-93} analyzes the following space: 
$$\Uc(L)=\{M\in K^{n\times \ell} \mid M \text{\ satisfies\ }(\ast)\}.$$
Since the condition ($\ast$) is determined in terms of vanishing or 
non-vanishing of minors of $M$, it follows that $\Uc(L)$ is an algebraic variety.
We define the \df{moduli space} $\Vc_K(L)$ of arrangements whose intersection lattice is $L$ as 
\[ \Vc_K(L) := \PGL(n,K) \backslash (\Uc(L) / (K^\times)^\ell). \] 
For a lattice $L$, 
we write $\Aut(L)$ for the set of automorphisms of posets, i.e., the set of bijections preserving the relations in the poset. 
\end{defin} 

The intersection lattices of arrangements of lines with few double intersection points (except the pencil and near-pencil) have mostly a very small moduli space.

\begin{examp}
Let $\Ac$ be the reflection arrangement of an irreducible complex reflection group of rank three. Then the moduli space $\Vc_\CC(L(\Ac))$ is finite and these finitely many realizations of $L(\Ac)$ in $\Vc_\CC(L(\Ac))$ are Galois 
conjugate under automorphisms of the smallest field extension of $\QQ$ over which $L(\Ac)$ is realizable:
\begin{enumerate}
\item Let $\Ac$ be the reflection arrangement of type $B_3$. Then $\Vc_\CC(L(\Ac))$ consists of one point since any realization of $L(\Ac)$ over $\CC$ is the same up to projectivities.
\item Let $\Ac$ be the reflection arrangement of type $H_3$. Then $\Vc_\CC(L(\Ac))$ consists of two points
since there are two realizations of $L(\Ac)$ over $\CC$ up to projectivities; these two points are Galois conjugate under the automorphism $\sqrt{5}\mapsto -\sqrt{5}$.
\end{enumerate}
\end{examp}

\begin{oppro}
Is it true that the moduli space of an irreducible simplicial arrangement over the real numbers is always finite?
\end{oppro}

We will see in the last section that our algorithm produces examples of irreducible simplicial arrangements over $\CC$ with infinite moduli space.

\begin{defin}\label{genM}
Let $L$ be a matroid of rank three, for instance the intersection lattice of an arrangement of rank three.
Let us call the one-dimensional elements \df{points} and the two-dimensional elements (hyperplanes) \df{lines}.
We say that the lines $H_1,\ldots,H_n$ of $L$ \df{generate} $L$ if there is a sequence of points and lines
$U_1,\ldots,U_m$ of $L$ such that:
\begin{enumerate}
\item $U_i=H_i$ for $i=1,\ldots,n$.
\item For all $i>n$ there exist $j,k<i$ such that $U_i \in \{ U_j\cap U_k, U_j+U_k\}$.
\item Every line of $L$ is in $\{U_1,\ldots,U_m\}$.
\end{enumerate}
We write
\[ g(L) := \min \{ n\in\NN \mid L \text{ is generated by } n \text{ lines}\}. \]
\end{defin}

In other words, an intersection lattice is generated by lines $H_1,\ldots,H_n$ if all the lines in $L$ are obtained by inductively adding intersection points of two lines or lines through two points.

\begin{examp}
Let $\Bc$ be the reflection arrangement of type $B_3$. Then $g(L(\Bc))=4$.
Let $\Ac$ be the reflection arrangement of type $H_3$. Then $g(L(\Ac))=5$.
\end{examp}

\begin{lemma}
Let $L$ be a matroid of rank three generated by at most $4$ lines in general position.
Then $|\Vc_K(L)|\le 1$ for any field $K$.
\end{lemma}
\begin{proof}
Let $H_1,\ldots,H_n$, $n\le 4$ be lines generating $L$.
Since these lines generate $L$, any realization of $L$ is uniquely determined by a choice of $n$ hyperplanes and by the matroid structure $L$. But all realizations are equivalent up to projectivities because $H_1,\ldots,H_n$ are in general position.
\end{proof}

\begin{corol}\label{mat4Q} Let $L$ be a matroid of rank three. If $g(L)\le 4$ and $L$ is realizable over $\CC$, then $L$ is realizable over $\QQ$.
\end{corol}

\begin{oppro}
What is the relation between $g(L)$ and $|\Vc_K(L)|$ for arbitrary matroids $L$ of rank three?
\end{oppro}

For most experiments with arrangements of hyperplanes, it is useful to have an algorithm which computes realizations of matroids. Although it is known that this is a difficult problem (see for example \cite[8]{BLVSWZ}), in practice (and in rank three), using generating sets of lines one can compute the moduli space for sufficiently large matroids (compare \cite{p-C10b} or \cite{p-C12}):

\algo{ModuliSpace}{$L$, $K$}
{Compute the moduli space of a matroid of rank three}
{a matroid $L$ of rank three, a field $K$.}
{A pair of algebraic varieties $V$, $E$ such that $\Vc_K(L)\cong V\backslash E$.}
{\label{MS}
\item Choose a set of generating lines $G:=\{H_1,\ldots,H_n\}$ of $L$.
\item For a largest subset $S:=\{H_{i_1},\ldots,H_{i_k}\}$ of $G$ in general position, choose basis elements of $K^3$ as coordinate vectors. For the remaining $d=|G\backslash S|$ lines in $G\backslash S$, choose coordinate vectors consisting of variables in a polynomial ring $K[X_1,\ldots,X_{3d}]$.
\item Every triple of lines in $L$ gives a conditions on the determinant of the corresponding coordinate vectors, yielding varieties $V$ and $E$ as required depending on whether the determinant has to be zero or not.
}

\begin{remar}
There are many possible technical improvements to Algorithm {\tt ModuliSpace} but they are not relevant for the goals of this article.
\end{remar}

\section{Greedy algorithms}
\label{sectec}

\subsection{The prototype}
We begin with a naive version of a greedy algorithm to compute arrangements with a given property $P$. In practice, $P$ will be a map assigning a number to each matroid $L(\Ac)$ in such a way that $P$ is satisfied if and only if this number is $0$. For example, $\Ac$ is simplicial if and only if $\sigma(\Ac)=0$.

The very first step is to find an arrangement over a finite field:

\algo{GreedyArrFiniteField}{$P$, $n$, $q$}
{Greedy search for arrangements of hyperplanes over a finite field}
{$n\in\NN$, a field $\FF_q$, a property $P$.}
{An arrangement with $n$ lines in $\PFq$ with $P$ if the algorithm terminates.}
{\label{GA1}
\item Choose a random set of lines $\Ac \subseteq \PFq$, $|\Ac|=n$.
\item While $\Ac$ does not satisfy $P$:
\begin{itemize}
\item Choose two random points $p_1\ne p_2$ in $L_2(\Ac)$, such that the line $\ell$ through $p_1$ and $p_2$ is not in $\Ac$.
\item Let $\ell'\ne \ell$ be a random line of $\Ac$ and $\Ac':=(\Ac\cup \{\ell\}) \backslash \{\ell'\}$.
\item If $\Ac'$ ``is closer to satisfying $P$'' than $\Ac$, then $\Ac\leftarrow \Ac'$.
\end{itemize}
\item Print $\Ac$.
}

\begin{remar}
There are several reasons why finite fields are useful in this context.
\begin{enumerate}
\item The most important reason is that the above algorithm may produce matroids requiring interesting relations with the ``same'' algebraic equations in realizations over any other fields. For example, if one chooses $\FF_{241}$, then the algorithm will automatically consider matroids enforcing certain $n$-th roots of unity for $n$ a divisor of $240$.
\item The second reason is of technical nature: an implementation is much faster over a finite field than over some rational number field.
\item A third reason is that the random set of lines at the beginning is ``less random'' if the underlying field has few elements. With a small field, the probability of finding an interesting intersection lattice by chance is higher.
\end{enumerate}
\end{remar}

Algorithm {\tt GreedyArrFiniteField} produces an arrangement over $\FF_q$. But since we are interested in real (or complex) arrangements, we still need to compute realizations of its intersection lattice in characteristic $0$, for instance with Algorithm {\tt ModuliSpace}.

\begin{examp}
Choose a finite field $\FF_q$ with large $q$, for example $q=14639$.
Algorithm {\tt GreedyArrFiniteField} with $P=$``simplicial'', $n=6,\ldots,37$ and $q=14639$ recovers all known rational simplicial arrangements within a few minutes.
\end{examp}

\subsection{Number fields}

\begin{remar}
\label{remQ}
Now there is another problem with {\tt GreedyArrFiniteField}. Since lines are repeatedly replaced with lines through existing intersection points, this algorithm tends to replace the original random arrangement by an arrangement whose intersection lattice has a small number of generators in the sense of Definition \ref{genM}. But then in most cases, a set of generators with four elements is attained and thus the resulting matroid is realizable over $\QQ$ by Corollary \ref{mat4Q} (if it is realizable in characteristic zero).
\end{remar}

This is why the above algorithm will mostly find matroids with rational realizations and thus miss most of the interesting examples. To address this problem we just choose a subset $F$ of the lines which should never be removed. This way the algorithm will regularly add lines generated by $F$, hence if $F$ contains some ``irrational'' entries, they will remain all the time and it is more likely that these ``irrationalities'' are enforced by the resulting matroid structure:

\algo{GreedyArrFiniteFieldAlgebraic}{$P$, $n$, $q$, $w$, $g$}
{Greedy search for arrangements of hyperplanes over a finite field with given algebraic elements}
{$n\in\NN$, a field $\FF_q$, a property $P$, an element $w\in \FF_q$ which is a root of a given polynomial $g$.}
{An arrangement with $n$ lines in $\PFq$ with $P$ if the algorithm terminates, possibly such that the field of definition of its moduli space in characteristic zero contains roots of the polynomial $g$.}
{
\item Choose a random set of lines $\Ac \subseteq \PFq$, $|\Ac|=n$ such that the first five lines are defined by the coordinate vectors $F=\{(0:0:1),(0:1:0),(1:0:0),(1:1:1),(1:0:w)\}$ and such that the other lines have ``small'' coefficients (e.g.\ in $\{\pm a\pm b w \mid a,b \in \{0,1,2\}\}$).
\item While $\Ac$ does not satisfy $P$:
\begin{itemize}
\item Choose two random points $p_1\ne p_2$ in $L_2(\Ac)$, such that the line $\ell$ through $p_1$ and $p_2$ is not in $\Ac$.
\item Let $\ell'\ne \ell$ be a random line of $\Ac$ which is not one of the first five lines and $\Ac':=(\Ac\cup \{\ell\}) \backslash \{\ell'\}$.
\item If $\Ac'$ ``is closer to satisfying $P$'' than $\Ac$, then $\Ac\leftarrow \Ac'$.
\end{itemize}
\item Print $\Ac$.
}

\begin{examp}
Let $w_0\in\RR$ be a root of $X^2-X-1\in \QQ[X]$, i.e.\ the golden ratio or its conjugate.
Choose a finite field $\FF_q$ with large $q$, for example $q=14639$ and $w:=9420\in \FF_q$.
Then $w^2-w-1=0$ and we may view $w$ as a golden ratio for $\FF_q$.
Algorithm {\tt GreedyArrFiniteFieldAlgebraic} with $P=$``simplicial'', $n=15$, $q=14639$ and $w=9420$ will produce an arrangement with the same intersection lattice as the reflection arrangement of type $H_3$. This matroid is almost impossible to obtain with the first version of the algorithm since it mostly finds rational arrangements and $\QQ(w_0)$ is the field of definition of the arrangement of type $H_3$.
\end{examp}

\begin{remar}
\label{remtrans}
Of course, the ideas for {\tt GreedyArrFiniteFieldAlgebraic} only increase the chance of finding interesting examples, the algorithm fails in many cases.
In particular, if $q$ is too small then often the resulting matroid will not be realizable in characteristic zero or the coordinate corresponding to the entry ``$w$'' will be rational in a realization over $\CC$.
If we choose $q$ large enough, then it is important to choose an original random arrangement with ``small coordinates'' since otherwise the greedy search fails to improve the property.

As an example, a good choice is $q=55441$. Then all minimal polynomials of $\cos(\pi/(2n))$ for $n=4,\ldots,12$ have roots in $\FF_q$. Since these algebraic numbers appear in the infinite series and in most of the known sporadic simplicial arrangements, this $\FF_q$ is a good choice to recover known examples with up to say $40$ lines.

One is tempted to increase the size of the set $F$ in such a way that \emph{every} realization of its matroid requires the desired irrationalities. However in practice it turns out that this is too restrictive in most cases (because it predetermines too much from the arrangement) and that the greedy search then fails to reach the property $P$.
\end{remar}

\begin{remar}\label{manyq}
Since {\tt GreedyArrFiniteFieldAlgebraic} is restricted to arrangements with the particular algebraic root $w$, it is necessary to call this procedure with many different values of $w$.
To obtain evidence that possibly all known arrangements with a certain number of lines and property $P$ were found, it is not reasonable to ``guess'' the algebraic numbers a priori.

A good solution to this problem is to take many primes, for example all primes from $53$ to $5987$, and for each prime $q$ to consider every possible value $w\in \FF_q$. Most of these choices will find arrangements whose matroid is realizable over $\QQ$, but since these primes are not so large (compared to $55441$), the considered $w$ will often satisfy algebraic equations with small coefficients.

Note that the first prime we consider is $53$ because for arrangements with at most $50$ lines, primes greater or equal to $53$ will almost only produce arrangements whose matroids are realizable in characteristic zero (which is not the case at all if one chooses for example $q=11$).
This comes from the fact that there are $q^2+q+1$ hyperplanes in $\FF_q^3$ and this number has to be much larger than $50$ (see \cite{p-CG-13} for examples of simplicial arrangements over $\FF_q$ with $\Ac = 3q$).
\end{remar}

\begin{remar}[Symmetry]\label{symm}
It may seem surprising that the algorithm, depending on the property $P$, produces so many highly symmetric arrangements, i.e.\ whose matroids have large automorphism groups, although the set of lines is completely random at the beginning.
In the case of simpliciality, this could be regarded as a strong hint that simplicial arrangements are built together using smaller symmetric pieces like the arrangements from the infinite series.

Another explanation is that arrangements with few double points only exist with symmetry: if an arrangement has a large symmetry group and contains a point which is not a double point, then its orbit consists of several points which will neither be double points.
\end{remar}

\subsection{Infinite moduli space}
\ 

Another interesting variation of {\tt GreedyArrFiniteFieldAlgebraic} is to choose a $w\in \FF_q$ which is \emph{not} a root of a polynomial with small coefficients.
If $w$ only satisfies algebraic relations with large coefficients, then a realization of the resulting matroid over $\CC$ is likely to replace $w$ by a transcendental number.
But if a realization has a transcendental coordinate, then the matroid has less linear dependencies,
i.e.\ its moduli space will possibly be infinite.

\algo{GreedyArrFiniteFieldTranscendent}{$P$, $n$, $q$, $(w_1,\ldots,w_k)$}
{Greedy search for arrangements of hyperplanes over a finite field whose matroid has infinite moduli space in characteristic zero}
{$n\in\NN$, a field $\FF_q$, a property $P$, elements $w_1,\ldots,w_k\in \FF_q$ which do not satisfy algebraic relations with small coefficients.}
{An arrangement with $n$ lines in $\PFq$ with $P$ if the algorithm terminates, possibly such that the moduli space in characteristic zero has dimension up to $k$.}
{
\item Choose a random set of lines $\Ac \subseteq \PFq$, $|\Ac|=n$ such that the first four lines are defined by the coordinate vectors $(0:0:1),(0:1:0),(1:0:0),(1:1:1)$, such that the next $m$ lines have coordinates including $w_1,\ldots,w_k$, and such that the other lines have ``small'' coefficients (e.g.\ in $\{-2,-1,0,1,2\}$).
\item While $\Ac$ does not satisfy $P$:
\begin{itemize}
\item Choose two random points $p_1\ne p_2$ in $L_2(\Ac)$, such that the line $\ell$ through $p_1$ and $p_2$ is not in $\Ac$.
\item Let $\ell'\ne \ell$ be a random line of $\Ac$ which is not one of the first $m+4$ lines and $\Ac':=(\Ac\cup \{\ell\}) \backslash \{\ell'\}$.
\item If $\Ac'$ ``is closer to satisfying $P$'' than $\Ac$, then $\Ac\leftarrow \Ac'$.
\end{itemize}
\item Print $\Ac$.
}

\begin{examp}
For example, if $q=55441$ and $w=31816\in\FF_q$, then $w$ is ``quite transcendental'' because it is not a root of any polynomial of degree less than $10$ with coefficients in $\{-2,\ldots,2\}$.
\end{examp}

\section{Results and examples}
\label{sec:resu}

\subsection{Simpliciality}
Using the method proposed in Remark \ref{manyq} we obtain a large database of simplicial arrangements.
Notice that we have to call the algorithm with all these primes from $53$ to $5987$ and all possible values $w$ many times before no further matroids are found any more. More precisely, even with a highly optimized and parallelized program in {\sc C++} running on a cluster with $1024$ cores, the algorithm still finds new examples after a week of computations.

In the following sections, we exhibit some of the most interesting arrangements found within this experiment.

\subsection{A simplicial arrangement with $35$ lines}

Let $\omega\in \RR$ be a (real) root of $X^4 - 3X^3 + 3X^2 - 3X + 1$ and

$R:=\{(1,0,0)$, $(0,1,0)$, $(0,0,1)$, $(1,1,1)$, $(1,0,\omega)$, $(1,\omega^3-2\omega^2+\omega-1,0)$, $(0,1,-\omega+1)$, $(1,1,-\omega+1)$, $(1,-\omega+1,-\omega+1)$, $(1,\omega^2-\omega+1,\omega)$, $(1,-\omega+1,0)$, $(1,-\omega+1,\omega^2-2\omega+1)$, $(1,-\omega^3+3\omega^2-3\omega+2,-\omega^3+3\omega^2-2\omega+1)$, $(1,1,-\omega^2+2\omega)$, $(1,-\omega^3+2\omega^2-2\omega+2,1)$, $(1,-\omega^3+2\omega^2-2\omega+2,-\omega^3+2\omega^2-\omega+1)$, $(1,\omega^2-\omega+1,-\omega^3+2\omega^2-\omega+1)$, $(1,\omega^2-\omega+1,-\omega^3+3\omega^2-2\omega+1)$, $(3,-\omega^3+2\omega^2-2\omega+3,-\omega^2+2\omega+1)$, $(1,-\omega^3+2\omega^2-\omega+1,-\omega^2+2\omega)$, $(0,1,\omega^3-2\omega^2)$, $(1,-\omega^3+3\omega^2-2\omega+1,-\omega^3+\omega^2+3\omega-1)$, $(1,0,-\omega^2+2\omega-1)$, $(1,-\omega^3+2\omega^2-2\omega+2,\omega)$, $(1,-\omega^3+\omega^2-\omega+1,\omega)$, $(1,-\omega^3+2\omega^2-2\omega+2,0)$, $(1,-\omega^3+2\omega^2-\omega+1,-\omega^2+3\omega-1)$, $(1,0,-\omega^2+3\omega-1)$, $(0,1,\omega^3-2\omega^2-\omega+1)$, $(1,-\omega+1,-\omega^3+3\omega^2-2\omega+1)$, $(1,-\omega^3+2\omega^2-3\omega+2,-\omega^3+3\omega^2-2\omega+1)$, $(1,\omega^2+1,-\omega^2+2\omega)$, $(1,-\omega^3+3\omega^2-3\omega+2,-2\omega^3+5\omega^2-2\omega+1)$, $(1,-\omega^3+2\omega^2-\omega+1,-\omega^3+\omega^2+3\omega-1)$, $(1,-\omega^3+\omega^2-\omega+1,-\omega^2+2\omega)\}$.

Notice that the golden ratio is one of $\omega^3-3\omega^2+2\omega-1$ or $-\omega^3+3\omega^2-2\omega+2$ depending on the choice of $\omega$. Then viewed as linear forms the coordinate vectors in $R$ define an arrangement $\Ac$ with $35$ hyperplanes which is real and simplicial. The two possible real values for $\omega$ give arrangements with isomorphic oriented matroids; Figure \ref{sim35} is a picture of $\Ac$.
The automorphism group of the intersection lattice of $\Ac$ is the dihedral group with $20$ elements.
The characteristic polynomial of $L(\Ac)$ is $(t-1)(t^2-34t+305)$, hence it is not free.

This arrangement is related to the reflection arrangement of type $H_3$ (which is contained in $\Ac$). However, despite the appearance, it is substantially different in the sense that there is no easy way to construct $\Ac$ starting from the arrangement of type $H_3$, because the required field extension is larger; it is not clear whether $\omega$ has a natural interpretation in mathematics (like the golden ratio).

\subsection{Simplicial arrangements with 1-dimensional moduli space}

Apart from the newly found real simplicial arrangement with $35$ lines, the biggest surprise from this computation is the existence of combinatorially simplicial arrangements over $\CC$ whose matroids have an infinite moduli space. Simpliciality is an extremal property (in the real case, every chamber has the least number of walls), so a simplicial arrangement should be quite rigid and should not allow small modifications preserving the matroid. We still conjecture that the moduli space of the matroid of a real simplicial arrangement is always finite.
With our greedy algorithm, we have found
$11$
pairwise non isomorphic combinatorially simplicial matroids with 1-dimensional moduli space over $\CC$. There is one with $16$ lines, four of them have $21$ lines, and six of them have $23$ lines (see Section \ref{mod1dat} in the Appendix for the explicit matroids).

Let us consider the smallest example $L$ in detail. It has $16$ lines and $38$ points. The lines $1,\ldots,16$ contain the points

$\{1$, $29$, $30$, $31$, $32$, $33$, $34$, $35$, $36$, $37$, $38\}$,
$\{2$, $5$, $12$, $15$, $26$, $28$, $35\}$,
$\{2$, $11$, $21$, $22$, $23$, $25$, $36\}$,
$\{3$, $4$, $5$, $6$, $9$, $23$, $33\}$,
$\{3$, $7$, $10$, $11$, $16$, $26$, $30\}$,
$\{2$, $6$, $13$, $16$, $17$, $24$, $38\}$,
$\{2$, $8$, $9$, $10$, $18$, $19$, $31\}$,
$\{2$, $4$, $7$, $14$, $20$, $27$, $37\}$,
$\{3$, $12$, $19$, $20$, $21$, $24$, $29\}$,
$\{3$, $8$, $13$, $14$, $15$, $22$, $34\}$,
$\{3$, $17$, $18$, $25$, $27$, $28$, $32\}$,
$\{1$, $6$, $8$, $21$, $26$, $27\}$,
$\{1$, $5$, $10$, $17$, $20$, $22\}$,
$\{1$, $7$, $12$, $13$, $18$, $23\}$,
$\{1$, $4$, $15$, $16$, $19$, $25\}$,
$\{1$, $9$, $11$, $14$, $24$, $28\}$

respectively; this completely determines the matroid $L$.
This matroid is generated by $5$ lines, $g(L)=5$ and Algorithm {\tt ModuliSpace} yields a $1$-dimensional variety $V\setminus E$ defined by the ideal generated by $X_1^4 - 3 X_1^3 X_2 + 4 X_1^2 X_2^2 - 2 X_1 X_2^3 + X_2^4$ and some equations for $E$.
It is easy to see that every solution $(X_1,X_2)\in(\CC^\times)^2$ is such that the quotient $X_1/X_2$ lies in $\CC\setminus \RR$. On the other hand, the cases $X_1=0$ or $X_2=0$ are excluded by the equations for $E$, thus there is no point in the moduli space with real coordinates and there are no real realizations of $L$.

We can proceed similarly for the other example matroids with 1-dimensional moduli space. None of them admits a real realization.

\subsection{Simplicial arrangements over the complex numbers}

The large part of the matroids found in the experiment have no apparent interesting properties. Since we have found
$1318$
simplicial arrangements over $\CC$ with up to $50$ lines, we content ourselves with some statistics instead of a complete enumeration.

\subsubsection{Number of lines}
\begin{figure}
\includegraphics[width=0.8\textwidth]{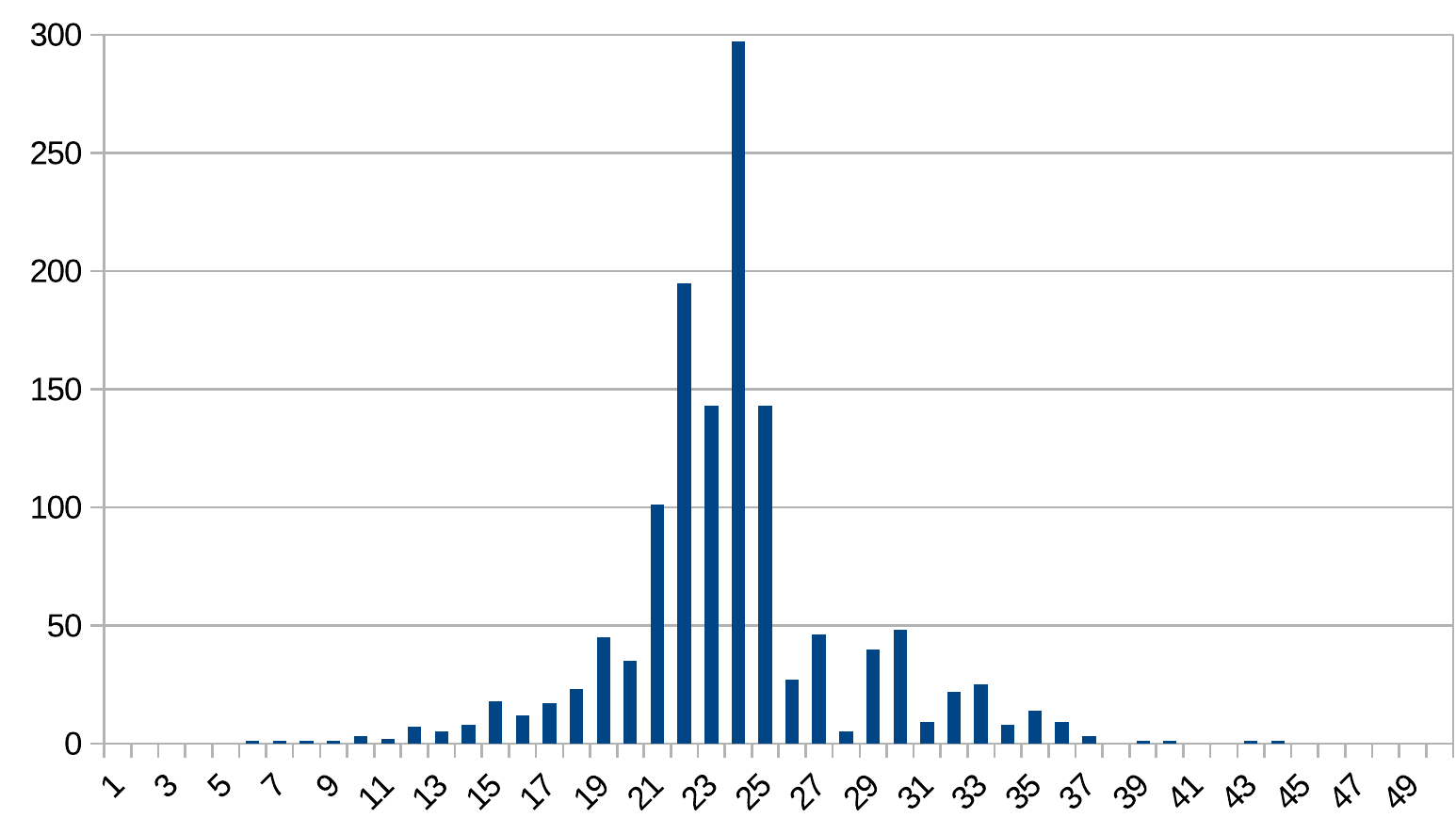}
\caption{Numbers of found combinatorially simplicial matroids up to isomorphisms which are realizable over $\CC$.\label{chart}}
\end{figure}

Figure \ref{chart} displays the numbers of combinatorially simplicial arrangements over $\CC$ that we found with the greedy algorithm.
\begin{remar}
\begin{enumerate}
\item The picture suggests that there will be a maximum for $24$ lines and that the set of simplicial arrangements could be finite when excluding the infinite series. However, we have to be careful with such a conjecture: for most of the choices of primes and of values $w$, the algorithm does not terminate in a short time (perhaps because there is no such arrangement). This is why we have to add a bound for the number of passes through the loop. Since this bound is the same for all numbers of lines, it is conceivable that this bound was too small for the higher numbers of lines and that we therefore missed many examples with more than $25$ lines.

On the other hand, the classification of crystallographic arrangements sets a precedent for a finite class of simplicial arrangements. I still conjecture that the set of non-supersolvable arrangements $\Ac$ with fixed $\sigma(\Ac)$ is finite in characteristic zero. In this case, Figure \ref{chart} could be quite close to the actual picture.
\item We did not find all the arrangements from the infinite series, probably because the required algebraic numbers are difficult to obtain in ``small'' finite fields. Thus it is very likely that we missed several other examples which require further algebraic relations.
\item It is unclear why the number of found arrangements with $23$, $26$, and $28$ lines is so small. A possible explanation is that these numbers do not allow matroids with many symmetries.
\end{enumerate}
\end{remar}

\subsubsection{Number fields}
We list here the number fields that appear as minimal fields of definition for the simplicial matroids in the database (except for those with infinite moduli space):
\begin{enumerate}
\item Cyclotomic fields: $\QQ(\zeta)$ for $\zeta$ a primitive $e$-th root of unity, $e\in\{3,4,5,7,8\}$.
\item\label{quadr} Quadratic fields: $\QQ(\sqrt{a})$ for $a\in \{2,3,5,-2,-7\}$.
\item\label{rezeta} The fields $\QQ(\zeta+\zeta^{-1})$ for $\zeta$ an $e$-th root of unity, $e\in\{7,9,11\}$.
\item The fields
\begin{eqnarray*}
K_1 &:=& \QQ[X]/(X^3-X+1), \\
K_2 &:=& \QQ[X]/(X^4 - 3X^3 + 3X^2 - 3X + 1)
\end{eqnarray*}
which have embeddings into $\RR$.
\item\label{4fields} The fields $\QQ[X]/(f)$ for $f$ in
\[ \{ 4X^4 + 8X^2 + 1, \quad
X^4 - 4X^3 + 8X^2 - 5X + 1, \]
\[ X^4 - 5X^3 + 11X^2 - 10X + 4, \quad
X^4 - X^3 + 2X + 1 \}. \]
\end{enumerate}

\begin{remar}
\begin{enumerate}
\item The fields $\QQ(\zeta_e+\zeta_e^{-1})$ for $\zeta_e$ a primitive $e$-th root of unity are also the minimal fields of definition of the arrangements of the infinite series (see \cite[Thm.\ 3.6]{p-C10b}). Moreover,
$\QQ(\sqrt{2})=\QQ(\zeta_8+\zeta_8^{-1})$, $\QQ(\sqrt{3})=\QQ(\zeta_{12}+\zeta_{12}^{-1})$, $\QQ(\sqrt{5})=\QQ(\zeta_{10}+\zeta_{10}^{-1})$, so one could move these fields from (\ref{quadr}) to (\ref{rezeta}).
\item The field $K_1$ is the minimal field of definition for $A(15,5)$ and $A(21,7)$ in the numbering by Gr\"unbaum \cite{p-G-09}. Note that this field was defined as $\QQ[X]/(X^3-3X+25)$ in \cite{p-C10b} (which is isomorphic to $K_1$), but the polynomial given here looks somewhat more natural.
\item The field $K_2$ is the minimal field of definition for the newly found arrangement with $35$ lines.
\item I have no good explanation for the four fields in (\ref{4fields}).
\item None of the extensions $K_1/\QQ$ nor $K_2/\QQ$ is Galois.
\end{enumerate}
\end{remar}

\subsubsection{Automorphism groups}
We find $47$ different groups of automorphisms of matroids up to isomorphisms. Their orders are
\[ 1, 2, 3, 4, 6, 8, 12, 16, 18, 20, 24, 32, 36, 40, 42, 48, 54, 64, \]
\[ 84, 96, 108, 110, 120, 128, 192, 200, 432, 600, 1536, 1764. \]
Since some of the (complex) reflection arrangements are combinatorially simplicial\footnote{More precisely, except the reflection arrangement of the group $G_{31}$, the reflection arrangement of a finite irreducible complex reflection group is combinatorially simplicial if and only if it is inductively free.} (see \cite{p-CG-13}), as for example the series of imprimitive reflection groups $G(e,1,r)$, these large groups appear in the list. Some of the other complex examples are probably related to reflection groups and inherit the symmetries.

\subsection{Free arrangements}
We would like to use our greedy algorithm with the property ``counterexample to Terao's conjecture'' (see \cite{OT} for details).
Since such a counterexample should be a matroid which has at least one free realization, we need a matroid whose characteristic polynomial has only integral roots. The degrees of a free arrangement are the roots $1,e,f$ of its characteristic polynomial. For an arbitrary arrangement, if $1,e,f$ are the roots of the characteristic polynomial, then the number $m(L):=(f-e)^2\in\ZZ$ (possibly negative if $e,f\in\CC\setminus\RR$) is easy to compute from the intersection lattice. We can call our greedy algorithm with the goal to optimize $m(L)\ge 0$, $m(L)$ a square. This way we obtain many matroids which may be filtered using several conditions (see for example \cite{p-Y-12} for a good overview). As a result we recover all the presently known examples of free but not inductively and free but not recursively free arrangements over $\CC$ including those with infinite moduli space (see \cite{p-CH-13}, \cite{p-ACKN-14}). We skip the details here because the experiment gave no new insight in this direction.

\subsection{$(n_k)$-configurations}
An $(n,k)$-configuration of lines and points consists of a set of $n$ lines and a set of $n$ points such that $k$ of the points are on each line and $k$ of the lines go through each point (see \cite{BP15} for previous results).
We have tried to find some interesting $(n_k)$-configurations with the greedy algorithm. It is in particular still open, whether a $(23_4)$-configuration exists. Although our algorithm comes very close to such a configuration (we find configurations with $23$ lines and points such that only one or two points lie on less than $4$ lines), the method presented in \cite{C17b} seems to be more adequate for this problem: it finds several complex arrangements which we do not obtain with the greedy algorithm.

\subsection{Few double points}
Of course one can run the greedy algorithm with the goal to minimize the number of double points. A short experiment in this direction produced exactly the known examples, amongst others the reflection arrangements of imprimitive reflection groups (including the Hesse configuration) and some of the simplicial arrangements from the infinite series.

\section{Appendix}

\subsection{Simplicial arrangements with 1-dimensional moduli space}\label{mod1dat}

We collect here the matroids of rank three that are simplicial and have 1-dimensional moduli space and that were found in the computation.
To reduce the amount of data we have chosen the following format:

Let $M$ be a matroid of rank three with $n$ lines.
Let $T=[t_1,t_2,\ldots]$ be the lexicographically ordered sequence of all triples $(a,b,c)$ with $1\le a<b<c\le n$, e.g.\ $n=4$ and $T=[(1,2,3),(1,2,4),(1,3,4),(2,3,4)]$.
We represent $M$ by the sequence of indices $i$ such that the lines labeled by $a,b,c$ for $(a,b,c)=t_i$ are linearly dependent.
\bigskip

{\tiny
\noindent
(1): $n=16$, $[$ $96$, $97$, $98$, $99$, $100$, $101$, $102$, $103$, $104$, $105$, $108$, $109$, $110$, $127$, $137$, $142$, $143$, $152$, $173$, $180$, $186$, $206$, $219$, $220$, $221$, $230$, $249$, $256$, $263$, $278$, $279$, $280$, $291$, $304$, $311$, $313$, $314$, $320$, $349$, $356$, $365$, $368$, $369$, $375$, $396$, $419$, $423$, $427$, $454$, $457$, $464$, $480$, $489$, $490$, $505$, $551$, $552$, $553$, $554$, $555$, $556$, $557$, $558$, $559$, $560$ $]$, \\ 
(2): $n=21$, $[$ $188$, $189$, $190$, $194$, $199$, $213$, $223$, $231$, $235$, $251$, $253$, $261$, $279$, $293$, $310$, $332$, $340$, $348$, $394$, $403$, $414$, $428$, $431$, $441$, $444$, $449$, $458$, $468$, $481$, $489$, $509$, $525$, $541$, $546$, $559$, $572$, $582$, $587$, $592$, $596$, $619$, $623$, $628$, $634$, $652$, $658$, $671$, $685$, $695$, $711$, $713$, $719$, $762$, $766$, $782$, $790$, $808$, $811$, $820$, $822$, $839$, $862$, $888$, $889$, $914$, $948$, $952$, $955$, $959$, $975$, $983$, $997$, $1007$, $1011$, $1060$, $1076$, $1082$, $1085$, $1095$, $1119$, $1129$, $1134$, $1139$, $1157$, $1201$, $1204$, $1207$, $1214$, $1222$, $1226$, $1229$, $1234$, $1251$, $1259$, $1279$, $1298$, $1313$, $1323$, $1330$ $]$, \\ 
(3): $n=21$, $[$ $188$, $189$, $190$, $201$, $209$, $225$, $241$, $246$, $259$, $269$, $271$, $279$, $291$, $305$, $319$, $330$, $349$, $368$, $374$, $386$, $393$, $409$, $410$, $445$, $453$, $457$, $465$, $473$, $485$, $497$, $530$, $536$, $543$, $552$, $556$, $561$, $579$, $601$, $612$, $626$, $634$, $658$, $661$, $670$, $687$, $693$, $719$, $741$, $745$, $765$, $772$, $778$, $803$, $812$, $837$, $850$, $862$, $864$, $871$, $894$, $910$, $913$, $916$, $933$, $942$, $958$, $974$, $990$, $998$, $1007$, $1009$, $1014$, $1021$, $1026$, $1029$, $1041$, $1063$, $1066$, $1074$, $1082$, $1088$, $1117$, $1122$, $1148$, $1151$, $1155$, $1159$, $1173$, $1177$, $1189$, $1218$, $1237$, $1251$, $1255$, $1275$, $1307$, $1314$, $1316$, $1330$ $]$, \\ 
(4): $n=21$, $[$ $188$, $189$, $190$, $195$, $196$, $211$, $239$, $248$, $250$, $271$, $301$, $306$, $313$, $318$, $327$, $351$, $354$, $374$, $387$, $388$, $420$, $421$, $424$, $451$, $469$, $473$, $477$, $487$, $503$, $509$, $518$, $520$, $538$, $544$, $567$, $574$, $586$, $592$, $611$, $620$, $628$, $635$, $651$, $683$, $694$, $714$, $739$, $742$, $743$, $756$, $767$, $793$, $794$, $798$, $825$, $828$, $839$, $841$, $853$, $863$, $866$, $888$, $899$, $901$, $904$, $914$, $922$, $926$, $934$, $946$, $961$, $983$, $991$, $998$, $1015$, $1021$, $1035$, $1068$, $1072$, $1085$, $1089$, $1093$, $1097$, $1104$, $1113$, $1126$, $1155$, $1157$, $1181$, $1191$, $1198$, $1214$, $1229$, $1235$, $1248$, $1271$, $1289$, $1306$, $1330$ $]$, \\ 
(5): $n=21$, $[$ $188$, $189$, $190$, $221$, $223$, $234$, $251$, $253$, $261$, $269$, $283$, $284$, $308$, $310$, $324$, $327$, $348$, $353$, $371$, $377$, $386$, $389$, $396$, $405$, $412$, $413$, $432$, $439$, $449$, $481$, $492$, $497$, $511$, $518$, $527$, $536$, $546$, $581$, $589$, $595$, $600$, $628$, $635$, $642$, $664$, $677$, $683$, $686$, $701$, $711$, $725$, $728$, $737$, $772$, $784$, $793$, $809$, $812$, $828$, $848$, $862$, $900$, $901$, $929$, $935$, $940$, $947$, $970$, $976$, $984$, $988$, $992$, $1002$, $1014$, $1033$, $1064$, $1080$, $1085$, $1090$, $1141$, $1142$, $1150$, $1154$, $1160$, $1170$, $1171$, $1176$, $1201$, $1211$, $1218$, $1225$, $1234$, $1237$, $1253$, $1267$, $1290$, $1306$, $1314$, $1330$ $]$, \\ 
(6): $n=23$, $[$ $185$, $186$, $231$, $366$, $370$, $375$, $382$, $383$, $384$, $400$, $405$, $406$, $407$, $412$, $414$, $419$, $425$, $431$, $432$, $433$, $436$, $456$, $460$, $462$, $473$, $485$, $506$, $513$, $520$, $533$, $537$, $547$, $557$, $567$, $607$, $625$, $629$, $639$, $647$, $653$, $658$, $667$, $702$, $710$, $716$, $733$, $741$, $755$, $765$, $776$, $796$, $815$, $818$, $830$, $843$, $850$, $851$, $878$, $908$, $911$, $934$, $937$, $948$, $957$, $959$, $977$, $983$, $988$, $1005$, $1029$, $1052$, $1071$, $1084$, $1090$, $1108$, $1123$, $1148$, $1156$, $1169$, $1173$, $1183$, $1186$, $1194$, $1207$, $1214$, $1215$, $1232$, $1249$, $1251$, $1277$, $1286$, $1295$, $1302$, $1342$, $1370$, $1375$, $1383$, $1392$, $1395$, $1412$, $1417$, $1425$, $1429$, $1445$, $1462$, $1474$, $1482$, $1500$, $1502$, $1510$, $1515$, $1528$, $1532$, $1541$, $1578$, $1595$, $1597$, $1598$, $1601$, $1608$, $1618$, $1620$, $1632$, $1638$, $1651$, $1670$, $1677$, $1688$, $1693$, $1699$, $1702$, $1720$, $1762$ $]$, \\ 
(7): $n=23$, $[$ $185$, $186$, $231$, $368$, $369$, $371$, $377$, $384$, $385$, $403$, $404$, $408$, $410$, $413$, $414$, $416$, $422$, $428$, $431$, $438$, $439$, $450$, $470$, $471$, $482$, $489$, $496$, $509$, $524$, $537$, $539$, $547$, $556$, $571$, $596$, $625$, $627$, $636$, $644$, $651$, $668$, $682$, $711$, $719$, $728$, $734$, $740$, $760$, $763$, $772$, $784$, $795$, $803$, $840$, $841$, $853$, $860$, $872$, $886$, $896$, $901$, $920$, $939$, $954$, $988$, $990$, $1007$, $1015$, $1033$, $1048$, $1055$, $1062$, $1070$, $1077$, $1079$, $1083$, $1104$, $1120$, $1126$, $1132$, $1134$, $1162$, $1163$, $1170$, $1177$, $1189$, $1202$, $1212$, $1245$, $1249$, $1255$, $1261$, $1262$, $1284$, $1292$, $1301$, $1305$, $1331$, $1348$, $1363$, $1384$, $1387$, $1403$, $1419$, $1420$, $1422$, $1432$, $1445$, $1463$, $1480$, $1521$, $1524$, $1526$, $1532$, $1540$, $1553$, $1568$, $1569$, $1591$, $1604$, $1625$, $1634$, $1642$, $1651$, $1666$, $1669$, $1685$, $1702$, $1705$, $1712$, $1717$, $1753$, $1758$ $]$, \\ 
(8): $n=23$, $[$ $175$, $176$, $231$, $366$, $367$, $370$, $389$, $392$, $396$, $397$, $400$, $408$, $416$, $420$, $423$, $424$, $425$, $435$, $436$, $437$, $439$, $445$, $451$, $467$, $475$, $487$, $491$, $497$, $517$, $546$, $547$, $564$, $573$, $581$, $599$, $604$, $630$, $636$, $650$, $666$, $682$, $688$, $707$, $709$, $726$, $739$, $741$, $768$, $778$, $791$, $794$, $805$, $818$, $820$, $863$, $883$, $889$, $897$, $901$, $908$, $917$, $920$, $940$, $941$, $971$, $975$, $999$, $1008$, $1011$, $1018$, $1043$, $1049$, $1059$, $1079$, $1122$, $1123$, $1141$, $1142$, $1146$, $1168$, $1174$, $1181$, $1189$, $1191$, $1202$, $1212$, $1213$, $1226$, $1263$, $1270$, $1287$, $1293$, $1297$, $1303$, $1306$, $1313$, $1317$, $1345$, $1347$, $1356$, $1368$, $1379$, $1381$, $1399$, $1420$, $1449$, $1454$, $1458$, $1473$, $1478$, $1494$, $1507$, $1510$, $1518$, $1536$, $1542$, $1566$, $1572$, $1580$, $1606$, $1613$, $1622$, $1626$, $1630$, $1645$, $1654$, $1665$, $1730$, $1746$, $1747$, $1749$, $1758$, $1768$ $]$, \\ 
(9): $n=23$, $[$ $200$, $202$, $227$, $340$, $344$, $350$, $354$, $356$, $361$, $364$, $370$, $371$, $381$, $382$, $390$, $396$, $398$, $403$, $418$, $426$, $427$, $458$, $460$, $466$, $477$, $485$, $490$, $508$, $517$, $518$, $529$, $547$, $564$, $581$, $587$, $606$, $630$, $641$, $653$, $655$, $678$, $681$, $683$, $694$, $708$, $713$, $730$, $750$, $761$, $795$, $801$, $813$, $816$, $822$, $825$, $845$, $848$, $860$, $867$, $888$, $893$, $919$, $937$, $943$, $963$, $964$, $985$, $991$, $1002$, $1023$, $1041$, $1056$, $1073$, $1085$, $1095$, $1106$, $1118$, $1123$, $1136$, $1156$, $1176$, $1191$, $1204$, $1222$, $1234$, $1240$, $1261$, $1272$, $1289$, $1304$, $1317$, $1321$, $1333$, $1356$, $1362$, $1381$, $1385$, $1392$, $1396$, $1411$, $1442$, $1447$, $1462$, $1472$, $1485$, $1491$, $1492$, $1527$, $1530$, $1537$, $1548$, $1553$, $1564$, $1569$, $1590$, $1592$, $1601$, $1611$, $1612$, $1636$, $1642$, $1649$, $1664$, $1671$, $1678$, $1696$, $1698$, $1705$, $1707$, $1711$, $1725$, $1733$, $1743$ $]$, \\ 
(10): $n=23$, $[$ $138$, $139$, $229$, $351$, $352$, $356$, $364$, $368$, $379$, $392$, $393$, $396$, $397$, $399$, $400$, $407$, $408$, $421$, $432$, $435$, $438$, $457$, $459$, $462$, $481$, $486$, $518$, $519$, $531$, $553$, $555$, $563$, $570$, $584$, $596$, $617$, $627$, $638$, $647$, $651$, $670$, $676$, $701$, $703$, $720$, $730$, $745$, $749$, $759$, $792$, $798$, $818$, $836$, $837$, $851$, $880$, $885$, $898$, $904$, $915$, $927$, $929$, $931$, $942$, $966$, $969$, $991$, $1007$, $1010$, $1025$, $1028$, $1032$, $1039$, $1059$, $1073$, $1079$, $1102$, $1105$, $1113$, $1120$, $1125$, $1139$, $1146$, $1172$, $1190$, $1199$, $1204$, $1212$, $1245$, $1258$, $1262$, $1281$, $1299$, $1301$, $1309$, $1315$, $1338$, $1347$, $1354$, $1386$, $1388$, $1405$, $1408$, $1412$, $1423$, $1435$, $1439$, $1447$, $1467$, $1474$, $1489$, $1504$, $1511$, $1522$, $1569$, $1570$, $1574$, $1606$, $1640$, $1642$, $1645$, $1648$, $1653$, $1654$, $1667$, $1687$, $1695$, $1710$, $1720$, $1723$, $1728$, $1750$, $1764$ $]$, \\ 
(11): $n=23$, $[$ $18$, $19$, $20$, $21$, $226$, $227$, $228$, $229$, $230$, $231$, $398$, $399$, $408$, $409$, $414$, $427$, $436$, $437$, $438$, $439$, $440$, $441$, $442$, $449$, $467$, $490$, $492$, $502$, $507$, $519$, $528$, $543$, $552$, $585$, $591$, $603$, $614$, $618$, $638$, $658$, $659$, $679$, $687$, $702$, $711$, $720$, $730$, $735$, $766$, $767$, $786$, $790$, $813$, $822$, $845$, $850$, $854$, $862$, $880$, $881$, $908$, $911$, $933$, $945$, $946$, $971$, $972$, $973$, $987$, $1021$, $1023$, $1026$, $1054$, $1056$, $1078$, $1083$, $1098$, $1105$, $1114$, $1115$, $1135$, $1149$, $1154$, $1171$, $1182$, $1188$, $1191$, $1212$, $1247$, $1255$, $1257$, $1278$, $1287$, $1297$, $1306$, $1344$, $1360$, $1371$, $1383$, $1387$, $1400$, $1412$, $1414$, $1421$, $1427$, $1439$, $1446$, $1455$, $1459$, $1470$, $1496$, $1507$, $1525$, $1535$, $1543$, $1567$, $1583$, $1587$, $1595$, $1610$, $1617$, $1620$, $1626$, $1630$, $1633$, $1645$, $1660$, $1697$, $1701$, $1730$, $1768$, $1769$, $1770$, $1771$ $]$,
}

\bibliographystyle{amsalpha}

\newcommand{\etalchar}[1]{$^{#1}$}
\def\cprime{$'$}
\providecommand{\bysame}{\leavevmode\hbox to3em{\hrulefill}\thinspace}
\providecommand{\MR}{\relax\ifhmode\unskip\space\fi MR }
\providecommand{\MRhref}[2]{%
  \href{http://www.ams.org/mathscinet-getitem?mr=#1}{#2}
}
\providecommand{\href}[2]{#2}

\end{document}